\documentclass [11pt]{article}
\usepackage{amsmath,amsthm,amssymb}
\usepackage{epsfig,graphics}

\newcommand{\Ndash}{\nobreakdash--}

\newtheorem{theorem}{Theorem}
\newtheorem{lemma}{Lemma}
\newtheorem*{coro}{Corollary}
\newtheorem*{claim}{Claim}

\theoremstyle{remark}
\newtheorem*{obs}{Observation}

\DeclareRobustCommand{\qedif}{\hbox{}\nobreak\quad\eqno\hbox{\qedsymbol}}

\numberwithin{equation}{section}

\title{Inner Functions with Derivatives in the Weak Hardy Space\\
{\small May 10, 2012}}
\author{Joseph A.\ Cima\\ Artur Nicolau\footnote{The second author is supported in part by the grants MTM2011-24606 and 2009SGR420.}}
\date{}

\begin{document}

\maketitle

\begin{abstract}
It is proved that exponential Blaschke products are the inner
functions whose derivative is in the weak Hardy space. Exponential
Blaschke products are described in terms of their logarithmic means
and also in terms of the behavior of the derivatives of functions in
the corresponding model space.
\end{abstract}

\section{Introduction}\label{section1}

For $0<p<\infty$, let $H^p$ be the Hardy space of analytic functions~$f$ in the unit disc~$\mathbb{D}$ of the complex plane for which
$$
\|f\|^{p}_{p} = \sup_{r<1} \int_{0}^{2\pi} |f(re^{i\theta})|^{p} \,d\theta<\infty.
$$
Any function $f\in H^{p}$ has radial limits at almost every point of
the unit circle $\partial\mathbb{D}$, that is,
$f(e^{i\theta})=\lim\limits_{r\to 1} f(re^{i\theta})$ exists a.e.\
$e^{i\theta}\in \partial\mathbb{D}$. An inner function~$I$ is a
bounded analytic function in~$\mathbb{D}$ satisfying $|I
(e^{i\theta})|=1$ a.e.\ $e^{i\theta}\in\partial\mathbb{D}$. Any
inner function~$I$ can be decomposed as $I=\psi B S$ where $\psi$ is
a unimodular constant,
$$
B(z)=\prod_{n}\frac{\overline{z}_{n}}{|z_{n}|}
\frac{z_{n}-z}{1-\overline{z}_{n}z},\quad z\in\mathbb{D},
$$
is a Blaschke product and
$$
S(z)=\exp \left( -\int^{2\pi}_{0} \frac{e^{i\theta}+z}{e^{i\theta}-z}\,d\mu (\theta)\right),\quad z\in\mathbb{D},
$$
is a singular inner function. Here $\{z_{n}\}$ are the zeros of $I$;
$d\mu$ is a positive singular measure and we use the convention
$\overline{z}/|z|=1$ if $z=0$. An inner function which extends
continuously to the closed unit disc must be a finite Blaschke
product, that is a Blaschke product with finitely many zeros. Hence,
the only inner functions $I$ such that $I'\in H^{1}$ are the finite
Blaschke products. Many authors have studied the problem of
determining the Hardy space~$H^{p}$, $0<p<1$, to which the
derivative of an inner function belongs. See \cite{AC}, \cite{Ah},
\cite{Cu}, \cite{CS}, \cite{FM}, \cite{GGJ}, \cite{GPV}, \cite{P1},
\cite{P2}, \cite{P3}, \cite{Pe}. Ahern and Clark proved that if an
inner function~$I$ satisfies $I'\in H^{1/2}$ then $I$ must be a
Blaschke product (\cite{AC}). Let $B$ be a Blaschke product with
zeros~$\{z_{n}\}$. Protas proved that the condition
$$
\sum_{n} (1 - |z_{n}|)^{1-p}<\infty,
$$
implies that $B'\in H^{p}$ if $1/2<p<1$ (see \cite{P1}). The converse is not true but Ahern proved that for $1/2<p<1$, $B'\in H^{p}$ if and only if there exists $a\in \mathbb{D}$ such that
$$
\sum (1-|w_{n}|)^{1-p}<\infty,
$$
where the sum is taken over all $w_{n}\in\mathbb{D}$ with
$B(w_{n})=a$. See Theorem~6.2 of \cite{Ah}. The paper~\cite{Ah} has
other very nice results in this direction but no geometrical
description of the Blaschke products~$B$ such that $B'\in H^{p}$,
$1/2<p<1$ in terms of the distribution of its zeros, is known.

For $0<p<\infty$, let $H^{p}_{w}$ be the weak Hardy space formed by
those analytic functions~$f$ in the unit disc for which there exists
a constant~$C=C(f)>0$ such that
$$
|\{ e^{i\theta}: |f(re^{i\theta})|>\lambda\}| \le C \lambda^{-p}
$$
for any $0<r<1$ and any $\lambda >0$. Here $|E|$ denotes the length of the set~$E\subset \partial\mathbb{D}$. Given an analytic function~$f$ in the unit disc, consider the non-tangential maximal function defined as
$$
Mf(e^{i\theta})=\sup \{|f(z)|: |z-e^{i\theta}|\le\alpha (1-|z|)\}
$$
where $\alpha>1$ is fixed. A fundamental result by Hardy and
Littlewood (for $1\le p<\infty$) and by Burkholder, Gundy and
Silverstein (for $0<p<1$) states that if $f$ is analytic in
$\mathbb{D}$ then $f\in H^{p}$ if and only if $Mf\in L^{p}
(\partial\mathbb{D} )$. See  \cite[p.~111]{Ga}. For $0<p<\infty$,
let $L^{p}_{w}(\partial\mathbb{D})$ be the weak $L^{p}$ space of
measurable functions~$f$ defined on $\partial\mathbb{D}$ for which
there exists a constant~$C=C(f)>0$ such that
$$
|\{ e^{i\theta}: |f(e^{i\theta})|>\lambda\}| \le C \lambda^{-p}
$$
for any $\lambda>0$. An analytic function~$f$ in the unit disc
belongs to $H^{p}_{w}$ if and only if $Mf\in L^{p}_{w}
(\partial\mathbb{D})$. See Remark 1 of (\cite{Al1}) or Theorem 2.1
of (\cite{Al2}). Functions in~$H^{p}_{w}$ have radial limits at
almost every point and the boundary values are in
$L^{p}_{w}(\partial\mathbb{D})$. One can show that an analytic
function~$f$ in the Smirnov class whose boundary values are in
$L^{p}_{w}(\partial\mathbb{D})$ belongs to $H^{p}_{w}$ (see
\cite{Al1} or \cite[p.~36]{CMR}).

In this paper we consider the extreme case~$p=1$ in the results by
Protas and Ahern mentioned above and we will see that the Hardy
space~$H^{p}$ should be replaced by the weak Hardy
space~$H^{1}_{w}$. It is worth mentioning that the Blaschke
products~$B$ for which $B'\in H^{1}_{w}$ can be described in terms
of the distribution of their zeros, as it is stated in
Theorem~\ref{theo1}  below.

A Blaschke product~$B$ is called an exponential Blaschke product if
there exists a constant~$M=M(B)>0$ such that for any $k=1,2,\dotsc$
one has $\#\{z: B(z)=0,\, 2^{-k-1} \le 1-|z|\le 2^{-k}\}\le M$. Let
$\{z_{n}\}$ be the zeros of $B$ ordered so that $|z_{n}|\le
|z_{n+1}|$, $n=1,2,\dotsc$ . Then $B$ is an exponential Blaschke
product if and only if there exist constants $c=c(B)>0$ and
$\delta=\delta(B)<1$ such that $1-|z_{n}|\le c \delta^{n}$, for any
$n=1,2,\dotsc$ .

\begin{theorem}\label{theo1}
Let $B$ be a Blaschke product. Then $B'\in H^{1}_{w}$ if and only if $B$ is an exponential Blaschke product.
\end{theorem}

Let $I$ be an inner function. Frostman
 (see~\cite{Fr}) tell us that there exists a set $E=E(I)$ of
 logarithmic capacity zero such that for any $a \in \mathbb{D} \setminus E$, the function  $(I-a)/(1 - \overline{a}I)$ is
 a Blaschke product. A Blaschke product $B$ is called indestructible
 if $E(B)= \emptyset$. This terminology was introduced in ~\cite{Ml} and further results can be found in
  ~\cite{Bi}, ~\cite{Mo}, ~\cite{Ro}. It is worth mentioning that no geometric
 description of indestructible Blaschke products in terms of the
 location of its zeros, is known. As a consequence of Theorem 1 we
 obtain that exponential Blaschke products are indestructible.

 \begin{coro}\label{cor1}
Let $B$ be an exponential Blaschke product. Then for any  $a \in
\mathbb{D}$ the function $ (I-a) /(1 - \overline{a}I) $ is also an
exponential Blaschke product.
\end{coro}

Another classical result of Frostman (see ~\cite{Fr}) tells that an
inner function $B$ is a Blaschke product if and only if
$$
 \lim_{r\to 1} \int^{2\pi}_{0} \log |B
(re^{i\theta})|\,d\theta=0.
$$
Exponential Blaschke products can be described in similar terms. For $0<r<1$, consider
$$
T(r)=T(B)(r)=\frac{1}{\log r} \frac{1}{2\pi} \int^{2\pi}_{0} \log
|B(re^{i\theta})|\,d\theta, \quad 0<r<1
$$
It is easy to show that finite Blaschke products are precisely the
inner functions for which $\sup\{T(r): r\in [0,1]\}<\infty$.
Exponential Blaschke products are the inner functions for which the
corresponding $T(r)$ has  a moderate growth, as it is stated in next
result.

\begin{theorem}\label{theo2}
Let $B$ be a Blaschke product. Then $B$ is an exponential Blaschke product if and only if there exists a constant~$M=M(B)>0$
such that $|T(1-2^{-N-1})-T(1-2^{-N})|\le M$ for any $N=1,2,\dotsc$ .
\end{theorem}

Given an inner function~$I$, let $(IH^{2})^{\perp}$ be the
orthogonal complement of the subspace~$IH^{2}$ in the Hardy
space~$H^{2}$. For $2/3<p<1$, W.~Cohn proved that $I'\in H^{p}$ if
and only if $f'\in H^{2p/(p+2)}$ for any $f\in (IH^{2})^{\perp}$.
See~\cite{Co}. We have the following version in the extreme
case~$p=1$.

\begin{theorem}\label{theo3}
Let $B$ be a Blaschke product. Then $B'\in H^{1}_{w}$ if and only if
there exists a constant~$C=C(B)>0$ such that for any $f\in
(BH^{2})^{\perp}$ and any $0<r<1$, one has
\begin{equation}\label{eq1.1}
|\{ e^{i\theta}: |f' (r e^{i\theta})| >\lambda \|f\|_{2}\}| \le C
{\lambda}^{-2/3}
\end{equation}
for any $\lambda>0$.
\end{theorem}

\section{Derivatives of exponential Blaschke products}\label{section2}

This section is devoted to the proof of Theorem~\ref{theo1}.

\begin{proof}[Proof of Theorem~\ref{theo1}]
\quad

\vspace*{7pt}

\noindent {\itshape Necessity.} Let $\{z_{n}\}$ be the zeros of $B$
ordered so that $|z_{n}|\le |z_{n+1}|_{1}$, $n=1,2,\dotsc$ . Assume
$B\in H^{1}_{w}$. Then
$$
|B'(\xi)|=\sum_{n}\frac{1-|z_{n}|^{2}}{|\xi -z_{n}|^{2}},\quad \text{a.e.\ }|\xi|=1
$$
(see \cite[Corollary~3]{AC}). We will prove that $B$ is an
exponential Blaschke product by contradiction. So, assume that there
exists a sequence of integers~$n_{k}$ with $\lim\limits_{k\to\infty}
(n_{k+1}-n_{k})=\infty$ such that $2^{-k-1} < 1-|z_n| \leq 2^{-k}$
for any $n=n_k , \ldots , n_{k+1}$
Let $J_{n}$ be the arc on the unit circle centered
at~$z_{n}/|z_{n}|$ of length~$2\pi (1-|z_{n}|)$. For $\xi\in J_{n}$
we have $|\xi-z_{n}|\le (\pi+1)(1-|z_{n}|)$. Hence
$$
|B'(\xi)| \ge \frac{1}{(\pi+1)^{2}}\frac{1}{1-|z_{n_{k}}|},\quad \xi\in F_{k}
$$
where
$$
F_{k}=\bigcup^{n_{k+1}}_{n=n_{k}}J_{n}.
$$
Since $B'\in H^{1}_{w}$, there exists a constant~$c_{1}>0$,
independent of~$k$, such that $ |F_{k}| \le c_{1} |J_{n_{k}}|$.
Since $\lim (n_{k+1}-n_{k})=\infty$ and $|J_{n}|\ge |J_{n_{k}}|/2$
for any $n_{k}\le n \le n_{k+1}$, there exists $\xi_{k} \in F_{k}$
such that the set of indices $ \mathcal{N}_{k}=\{n: n_{k}\le n\le
n_{k+1},\quad \xi_{k}\in J_{n}\} $ safisfies $\#
\mathcal{N}_{k}\to\infty$ as $k\to\infty$. Pick
$m_{k}\in\mathcal{N}_{k}$.
For any $n\in \mathcal{N}_{k}$ and any $\xi\in J_{m_{k}}$ we have
$|\xi - \xi_{k}|\le 2\pi (1-|z_{m_{k}}|)\le 4\pi (1-|z_{n}|)$ and
$|\xi_{k}-z_{n}|\le (\pi+1) (1-|z_{n}|)$. Hence $|\xi-z_{n}|\le (5
\pi +1)(1-|z_{n}|)$. We deduce that for almost every $\xi \in
J_{m_{k}}$ one has
$$
|B'(\xi)|\ge\sum_{n\in\mathcal{N}_{k}} \frac{1-|z_{n}|^{2}}{|\xi-
z_{n}|^{2}}\ge \frac{1}{(5\pi+1)^{2}} \sum_{n\in\mathcal{N}_{k}}
\frac{1}{1-|z_{n}|} \ge \frac{1}{ 2(5\pi+1)^{2}}
\frac{\#\mathcal{N}_{k}}{1-|z_{m_{k}}|}.
$$
Since $|J_{m_{k}}| =2\pi (1-|z_{m_{k}}|)$, the fact that
$\#\mathcal{N}_{k}\to\infty$ as $k\to\infty$, contradicts that
$B'\in H^{1}_{w}$.

\vspace*{7pt}

The proof of the {\it sufficiency} uses the following auxiliary result.

\begin{lemma}\label{lem1}
Fix $\mu>10$. Let $\{w_{k}\}$ be a sequence of points in the unit disk ordered so that $|w_{k}|\le |w_{k+1}|$, $k=1,2,\dotsc$, satisfying
\begin{enumerate}
\item[(a)] $1-|w_{1}| \le 1/\mu$.
\item[(b)] There exists $N>0$ such that
$$
C_{0}=\sup_{k} \frac{1-|w_{k+N}|}{1-|w_{k}|}<1.
$$
\end{enumerate}
Then there exist a sequence of numbers $n_{k}$, $n_{k}\nearrow
\infty$ and a constant~$K=K(C_{0},N)$ such that
\begin{enumerate}
\item[(c)] $\sum\limits_{k} 2^{n_{k}} (1-|w_{k}|)\le K/\mu$,
\item[(d)] $\sum\limits_{k} 2^{-2n_{k}} (1-|w_{k}|)^{-1}\le \mu$.
\end{enumerate}
\end{lemma}

\begin{proofwqed}[Proof of Lemma~\ref{lem1}]
Choose $n_{k}$ satisfying
$$
\frac{1}{2^{2n_{k}}(1-|w_{k}|)} =\frac{\mu}{100 k^{2}}.
$$
Thus (d) holds. Since $2^{2n_{k}}=100k^{2}/\mu (1-|w_{k}|)$, we have
$$
\sum_{k} 2^{n_{k}} (1-|w_{k}|)=\frac{10}{\mu^{1/2}} \sum_{k} k (1-|w_{k}|)^{1/2}.
$$
Condition~(b) gives that $\{w_{k}\}$ may be split into at most
$N$~geometric progressions where, by~(a), first term is smaller
than~$1/\mu $. Hence
$$
\sum_{k} k(1-|w_{k}|)^{1/2}\le \frac{K(C_{0},N)}{\mu^{1/2}}.\qedif
$$
\end{proofwqed}

We now prove the converse direction in Theorem~\ref{theo1}. So let
$B$ be an exponential Blaschke product. The result of D.~Protas
mentioned at the Introduction  (\cite{P1}) gives that $B'\in H^{p}$
for any~$p<1$. Moreover Theorem~\ref{theo2} of~\cite{AC} gives
$$
|B'(\xi)|=\sum_{n}\frac{1-|z_{n}|^{2}}{|\xi-z_{n}|^{2}},\quad \text{a.e.\ }\xi\in \partial\mathbb{D}.
$$
Here $\{z_{n}\}$ is the sequence of zeros of $B$ ordered so that $|z_{n}|\le |z_{n+1}|$, $n=1,2,\dotsc$ . So, it will suffice to show that there exists a constant~$C>0$ such that for any $\lambda>0$ one has
\begin{equation}\label{eq2.1}
\left|\left\{ e^{i\theta}: \sum_{n}\frac{1-|z_{n}|^{2}}{|e^{i\theta}-z_{n}|^{2}}>\lambda\right\} \right| \le \frac{C}{\lambda}.
\end{equation}

Fix $\lambda>0$ and consider the set
$E=E(\lambda)=\{k\in\mathbb{N}:|z_{k}|<1-M\lambda^{-1}\}$, where
$M=M(\{z_{n}\})$ is a number depending on the sequence~$\{z_{n}\}$
which will be fixed later. Since $|\xi-z_{k}|\ge 1-|z_{k}|$ for any
$\xi\in \partial\mathbb{D}$, we have
$$
\sum_{k\in E} \frac{1-|z_{k}|^{2}}{|\xi-z_{k}|^{2}}\le 2\sum_{k\in E} \frac{1}{1-|z_{k}|}.
$$
Since $B$ is an exponential Blaschke product, there exist a constant
$N=N(B)$ such that for any $l=1, 2, \ldots$, the number of points in
$\{z_{n}\}$ with $2^{-l} \leq 1-|z_n| \leq 2^{-l+1}$ is smaller than
$N$. So
$$
\sum_{k\in E} \frac{1}{1-|z_{k}|} \le  \frac{ N\lambda}{M}.
$$
Choose $M=M (\{z_{n}\})=4N$ to deduce that for any
$\xi\in\partial\mathbb{D}$ one has
\begin{equation}\label{eq2.2}
\sum_{k\in E} \frac{1-|z_{k}|^{2}}{|\xi -z_{k}|^{2}}\le\frac{\lambda}{2}.
\end{equation}
Apply Lemma~\ref{lem1} to the sequence~$\{w_{k}\}=\{z_{k}: k\notin
E\}$ and the parameter~$\mu=\lambda/4N$ to get numbers~$\{n_{k}\}$
satisfying (c) and (d). As before, let  $J_k$ be the arc on the unit
circle centered at $z_k / |z_k|$ of length $2 \pi (1- |z_k|)$.
Observe that
\begin{equation}\label{eq2.3}
\left\{ e^{i\theta}:\sum_{k\notin E}
\frac{1-|z_{k}|^{2}}{|e^{i\theta}-z_{k}|^{2}}
>\frac{\lambda}{2}\right\} \subseteq \bigcup_{k\notin E} N^{-1/2} 2^{n_{k}}
J_k .
\end{equation}
Actually, if $e^{i\theta} \notin N^{-1/2} 2^{n_{k}}I(z_{k})$ one has
$|e^{i\theta} -z_{k}|\ge N^{-1/2} 2^{n_{k}} (1-|z_{k}|)$ and it
follows that
$$
\sum_{k\notin E} \frac{1-|z_{k}|^{2}}{|e^{i\theta}-z_{k}|^{2}} \le
2N \sum_{k\notin E} \frac{1}{2^{2n_{k}}(1-|z_{k}|)}
$$
which by (d) is bounded by  $\lambda/2$. So, \eqref{eq2.3} holds. Now observe that (c) gives that
$$
\left| \bigcup_{k\notin E} 2^{n_{k}} J_k \right| \le 2 \pi
\sum_{k\notin E} 2^{n_{k}} (1-|z_{k}|) \le 2 \pi
\frac{4KN}{\lambda}.
$$
Hence, applying \eqref{eq2.2} and \eqref{eq2.3} we deduce \eqref{eq2.1} and the proof is completed.
\end{proof}

It is worth mentioning that there exists no infinite Blaschke product~$B$ with $B'\in H^{1}_{w}$ such that
$$
\lim_{\lambda\to\infty} \lambda |\{ e^{i\theta}:|B' (e^{i\theta})|
>\lambda\}|=0.
$$
Actually if $\{z_{n}\}$ are the zeros of~$B$, since
$$
|B' (e^{i\theta})|=\sum_{n} \frac{1-|z_{n}|^{2}} {|e^{i\theta} -z_{n}|^{2}} \quad\text{a.e.\ } e^{i\theta} \in\partial\mathbb{D},
$$
one deduces $|B' (e^{i\theta})| \ge 1/4(1-|z_{n}|)$ for any $e^{i\theta}\in \partial\mathbb{D}$ with $|e^{i\theta}-z_{n}| \le 2 (1-|z_{n}|)$.

\section{A Frostman type result}\label{section3}

This section is devoted to the proof of Theorem~\ref{theo2}.

Let $B$ be a Blaschke with zeros~$\{z_{n}\}$. Fix $0<r<1$. Recall the following classical calculation,
\begin{equation*}
\begin{split}
\frac{1}{2\pi} \int_{0}^{2\pi} \log |B (re^{i\theta})| \,d\theta&= \sum_{n} \frac{1}{2\pi} \int^{2\pi}_{0} \log \left| \frac{re^{i\theta}-z_{n}}{1-\overline{z}_{n} re^{i\theta}}\right|\,d\theta\\*[7pt]
&=(\log r)\# \{z_{n}: |z_{n}| \le r\}+ \sum_{n:|z_{n}|\ge r} \log |z_{n}|.
\end{split}
\end{equation*}
Using the notation
$$
T(r)=\frac{1}{\log r} \frac{1}{2\pi} \int^{2\pi}_{0} \log |B(re^{i\theta})|\,d\theta,
$$
we have
\begin{equation}\label{eq3.1}
T(r)=\# \{z_{n}: |z_{n}| \le r\} +\frac{1}{\log r} \sum\limits_{n:
|z_{n}|\ge r}\log |z_{n}|.
\end{equation}
Observe that $\sup\{T(r): r\in [0,1]\}<\infty$ if and only if $B$ is a finite Blaschke product.

\begin{proof}[Proof of Theorem~\ref{theo2}]
Assume that $B$ is an exponential Blaschke product. We will use the
decomposition of $T(r)$ given in~\eqref{eq3.1}. Observe that there
exists a constant~$C>0$ such that for $1/2\le r <1$, one has
$$
\frac{1}{\log r} \sum\limits_{n: |z_{n}|\ge r}\log |z_{n}| \le C
\frac{1}{1- r} \sum\limits_{n: |z_{n}|\ge r}1- |z_{n}| .
$$
Since $B$ is an exponential Blaschke product, its zeros~$\{z_{n}\}$,
ordered so that $|z_{n}|\le |z_{n+1}|$, $n=1,2,\dotsc$,  satisfy
\begin{equation}\label{eq3.2}
\sup_{n}\frac{1-|z_{n+K}|}{1-|z_{n}|} =\alpha<1
\end{equation}
for a certain fixed integer $K>0$. Thus
$$
\frac{1}{1- r} \sum\limits_{n: |z_{n}|\ge r} 1- |z_{n}| \le
\frac{K}{1-\alpha},\quad 0<r<1.
$$
Therefore
$$
\frac{1}{\log r} \sum\limits_{n: |z_{n}|\ge r}\log |z_{n}| \le
\frac{CK}{1-\alpha}
$$
for any $1/2\le r <1$.
So, we deduce that
$$
|T (1-2^{-N-1})-T(1-2^{-N})|\le \# \{ z_{n}: 1-2^{-N}\le |z_{n}| \le
1-2^{-N-1}\} +\frac{2CK} {1-\alpha}
$$
for any $N=1,2,\dotsc$, which is uniformly bounded because $B$ is an
exponential Blaschke product.

Now let us show the converse. Let $\{z_{n}\}$ be the zeros of the Blaschke product~$B$. Given $1/2<r<1$, choose $r_{1}$ such that $\log r_{1}^{-1}=(\log r^{-1})/2$. Using \eqref{eq3.1} one has
\begin{equation*}
\begin{split}
T(r_{1})-T(r) &= \# \{z_{n}:r \le |z_{n}|\le r_{1}\}
 + \frac{2\sum\limits_{n:|z_{n}|\ge r_{1}}\log |z_{n}|^{-1}- \sum\limits_{n:|z_{n}|\ge r}\log |z_{n}|^{-1}}{\log r^{-1}}\\*[7pt]
&=\# \{z_{n}: r\le |z_{n}|\le r_{1}\}-\frac{\sum\limits_{n: r_{1}\ge |z_{n}|\ge r}\log |z_{n}|^{-1}}{\log r^{-1}}
 +\frac{\sum\limits_{n:  |z_{n}|\ge r_{1}}\log |z_{n}|^{-1}}{\log r^{-1}}.
\end{split}
\end{equation*}
Since
$$
\sum_{r_{1}\ge |z_{n}| \ge r}\log |z_{n}|^{-1} \le (\log r^{-1})\# \{z_{n}: r\le |z_{n}|\le r_{1}\}
$$
we have
$$
\#\{ z_{n}: r\le |z_{n}| \le r_{1}\} -\frac{\sum\limits_{r_{1}\ge  |z_{n}|\ge r}\log |z_{n}|^{-1}}{\log r^{-1}}\ge 0.
$$
Hence
$$
T(r_{1})-T(r)\ge \frac{\sum\limits_{n:  |z_{n}|\ge r_{1}}\log |z_{n}|^{-1}}{\log r^{-1}}.
$$
The estimate in the hypothesis and our choide of $r_{1}$ gives that there exists a constant~$C>0$ independent or $r$ such that $|T(r_{1})-T(r)| \le C$. We deduce that
$$
\frac{\sum\limits_{n:  |z_{n}|\ge r_{1}}\log |z_{n}|^{-1}}{\log r^{-1}}\le C .
$$
Pick $r_{2}$ such that $\log r_{2}^{-1} =(\log r^{-1})/8$ and observe
$$
\# \{ z_{n}: r_{1}\le |z_{n}|\le r_{2}\} \le 8 \frac{\sum\limits_{n: r_{2}\ge |z_{n}|\ge r_{1}}\log |z_{n}|^{-1}}{\log r^{-1}}\le 8C.
$$
Now, given $N >0$ pick $r= (1-2^{-N})^{2}$. Then $r_{1}=1-2^{-N}$
and $r_{2}=(1-2^{-N})^{1/4}$. Since $r_{2}\ge 1-2^{-N-1}$ if $N$~is
large enough, we deduce that
$$
\# \{z_{n}: 1-2^{-N}\le |z_{n}| \le 1-2^{-N-1} \} \le \# \{z_{n}:
r_{1}\le |z_{n}| \le r_{2} \}\le 8C.
$$
Applying Theorem~\ref{theo1} we deduce that $B'\in H^{1}_{w}$.
\end{proof}

\section{Derivatives of functions orthogonal to invariant subspaces}\label{section4}

This section is devoted to the proof of Theorem~\ref{theo3}.

\begin{proof}
Assume $B'\in H^{1}_{w}$. Let $\{z_{k}\}$ be the zeros of $B$
ordered so that $|z_{k}|\le |z_{k+1}|$, $k=1,2,\dotsc$ . According
to Theorem~\ref{theo1}, there exists an integer $N>0$ such that
\begin{equation}\label{eq4.1}
\sup_{k}\frac{1-|z_{k+N}|} {1-|z_{k}|} <1
\end{equation}
So, $\{z_{k}\}$ can be split into a finite union
$\{z_{k}\}=\Lambda_{1}\cup \dotsb \cup \Lambda_{m}$, $m=m(N)$, of
sequences~$\Lambda_{j}$ satisfying
$$
\sup_{k: z_{k}\in \Lambda_{j}}
\frac{1-|z_{k+1}|}{1-|z_{k}|}<\frac{1}{2}, \quad \text{for any
}j=1,\dotsc,m.
$$
Let $B_{j}$ be the Blaschke product with zeros~$\Lambda_{j}$. Since
$B=B_{1},\dotsc,B_{m}$, it is enough to prove the estimate for any
$f\in (B_{i}H^{2})^{\perp}$, $i=1,\dotsc,m$. In other words, one can
assume that $N=1$ in equation~\eqref{eq4.1}. So, assume
 \begin{equation}\label{eq4.2}
 \sup_{k} \frac{1-|z_{k+1}|} {1-|z_{k}|} <\frac{1}{2}.
 \end{equation}
 Let $f(z)=\sum\limits^{M}_{k=1} \beta_{k} (1-|z_{k}|)^{1/2}/(1-\overline{z}_{k}z)\in (BH^{2})^{\perp}$. Since \eqref{eq4.2} holds, $\{z_{k}\}$ is an interpolating
 sequence. Then there exists a constant $C=C(\{z_n\})$ such that
$$
C^{-2} \sum^{M}_{k=1}|\beta_{k}|^{2} \le \|f\|^{2}_{2} \le C^2
\sum^{M}_{k=1}|\beta_{k}|^{2}
$$
(see Theorem~B in \cite{Co}). Hence
\begin{equation*}
\begin{split}
|f'(e^{i\theta})| &\le \sum^{M}|\beta_{k}| \frac{(1-|z_{k}|)^{1/2}}{|1-\overline{z}_{k}e^{i\theta}|^{2}} \le \left(\sum^{M}_{k=1} |\beta_{k}|^{2}\right)^{1/2} \left(\sum^{M}_{k=1} \frac{1-|z_{k}|} {|1-\overline{z}_{k} e^{i\theta}|^{4}}\right)^{1/2}\\*[7pt]
&\le C\|f\|_{2} \left( \sum^{M}_{k=1} \frac{1-|z_{k}|} {|1-\overline{z}_{k}e^{i\theta}|^{4}}\right)^{1/2}
\end{split}
\end{equation*}


\noindent Fix $\lambda>0$. We have
$$
\{e^{i\theta}: |f' (e^{i\theta})| >\lambda\} \subseteq \left\{
e^{i\theta}: \sum^{M}_{k=1} \frac{1-|z_{k}|}{|1-\overline{z}_{k}
e^{i\theta}|^{4}} > \frac{\lambda^{2}}{C^{2}\|f\|_{2}^{2}} \right\}.
$$
Let $k_{0}$ be the largest integer between~$1$ and $M$ such that
$$
\frac{4}{(1-|z_{k_{0}}|)^{3}}\le \frac{\lambda^{2}}{C^{2}\|f\|^{2}_{2}}.
$$
Since we can assume $\lambda$ is large, the number~$k_{0}$ exists. Observe that by~\eqref{eq4.2},
$$
\sum^{k_{0}}_{k=1} \frac{1-|z_{k}|} {|1-\overline{z}_{k}
e^{i\theta}|^{4}} \le \sum^{k_{0}}_{k=1} \frac{1}{(1- |z_{k}|)^{3}}
\le \frac{2}{(1-|z_{k_{0}}|)^{3}}.
$$
Hence
$$
\{e^{i\theta} :| f' (e^{i\theta})| >\lambda\} \subseteq \left\{ e^{i\theta} :\sum^{M}_{k_{0}+1} \frac{1-|z_{k}|} {|1-\overline{z}_{k} e^{i\theta}|^{4}} >\frac{\lambda^{2}}{2C^{2}\|f\|^{2}_{2}}\right\}.
$$
Pick $N_{k}>0$ satisfying
$$
\frac{1}{N_{k}^{4}(1-|z_{k}|)^{3}} = \frac{\lambda^{2}} {10C^{2} \|f\|^{2}_{2} (k-k_{0})^{2}}, \quad k=k_{0}+1,\dotsc,M.
$$
Let $N_{k}I_{k}$ denote the arc on the unit circle centered at~$z_{k}/|z_{k}|$ of length~$2N_{k}(1-|z_{k}|)$. We claim that
 \begin{equation}\label{eq4.3}
 \left\{ e^{i\theta}: \sum^{M}_{ k_{0}+1} \frac{1-|z_{k}|} {|1-\overline{z}_{k}e^{i\theta}|^{4}} >\frac{\lambda^{2}} {2C^{2}\|f\|^{2}_{2}}\right\} \subseteq \bigcup^{M}_{k_{0}+1} N_{k}I_{k}.
 \end{equation}
Actually if $e^{i\theta} \notin \bigcup\limits^{M}_{k_{0}+1} N_{k}I_{k}$, we have
$$
\sum^{M}_{k_{0}+1} \frac{1-|z_{k}|} {|1-\overline{z}_{k}e^{i\theta}|^{4}} \le \sum^{M}_{k_{0}+1} \frac{1}{ N_{k}^{4}(1-|z_{k}|)^{3}}\le \frac{\lambda^{2}}{2C^{2}\|f\|^{2}_{2}}.
$$
So, \eqref{eq4.3} holds. We deduce that
\begin{equation*}
\begin{split}
|\{ e^{i\theta}  : |f'(e^{i\theta})| >\lambda\} | &\le
\sum^{M}_{k_{0}+1} N_{k} (1-|z_{k}|)\\*[7pt]  &=
\frac{10^{1/4}C^{1/2} \|f\|_{2}^{1/2}}{ \lambda^{1/2}}
\sum^{M}_{k_{0}+1} (k-k_{0})^{1/2} (1-|z_{k}|)^{1/4}.
\end{split}
\end{equation*}
Now \eqref{eq4.2} and the choice of $k_{0}$ gives
\begin{equation*}
\begin{split}
\sum^{M}_{k_{0}+1} (k-k_{0})^{1/2} (1-|z_{k}|)^{1/4}&\le 10 (1-|z_{k_{0}+1}|)^{1/4}\\
&\le 10\left(\frac{4C^{2}\|f\|^{2}_{2}}{\lambda^{2}}\right)^{1/12}.
\end{split}
\end{equation*}
We deduce
$$
|\{ e^{i\theta} :|f' (e^{i\theta})| >\lambda\}| \le \frac{20C^{2/3}\|f\|_{2}^{2/3}} {\lambda^{2/3}}
$$
for any $f$ which is a finite linear combination of $(1-|z_{k}|)^{1/2}/ (1-\overline{z}_{k}z)$. Since these
 functions are dense in $(BH^{2})^{\perp}$ we deduce that \eqref{eq1.1} holds.

\vspace*{7pt}

\noindent Let us now prove the converse. Let $\{z_{n}\}$ be the
sequence of zeros of~$B$. For $m=1,2,\dotsc$, let $E_{m} $ be the
annuli~$E_{m}= \{z: 1-2^{-m}\le |z|\le 1-2^{-m-1}\}$. We will show
that there exists a constant~$K>0$ such that $\# \{z_{n}: z_{n}\in
E_{m}\} \le K$ for any $m=1,2,\dotsc$ . Then, according to
Theorem~\ref{theo1}, it would follow that $B'\in H^{1}_{w}$. Fix $m
\geq 1$. Split $E_{m}$ into $2^{m}$ truncated sectors~$Q_{j}$, that
is, $E_{m}=\bigcup\limits^{2^{m}-1}_{j=0} Q_{j}$, where
$$
Q_{j}=\{ z= re^{i\theta} \in E_{m}: |\theta-2\pi j 2^{-m}| <\pi 2^{-m}\},\quad j=0,\dotsc, 2^{m}-1.
$$
The proof is organized in two steps. First we will show that there
exists at most a fixed number (independent of~$m$) of
sectors~$Q_{j}$ which contain a point of the sequence~$\{z_{n}\}$.
Second, we will show that each~$Q_{j}$ can contain at most a fixed
number (independent of~$j$ and~$m$) of points of the
sequence~$\{z_{n}\}$.

Let us group the sectors~$\{Q_{j}\}$ into ten families~$G_{\ell}$, $\ell=1,\dotsc,10$, defined as $G_{\ell}=\{Q_{j}\}_{j}$ where the index~$j$ runs over all indices~$j$ such that $j=\ell$ (mod.~$10$), that is, $j=\ell+k10$ for a certain integer~$k$. See Figure~\ref{fig1}.
\begin{figure}[h]
\begin{center}
\epsfig{file=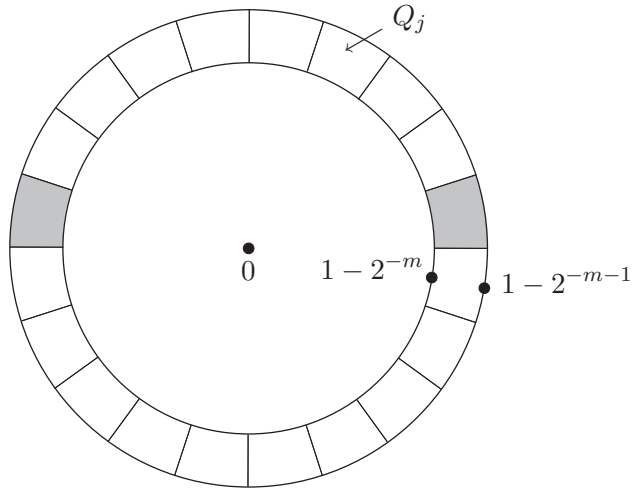} \\*[-5pt]
\caption{$G_{1}$ consists of the two shadowed sectors.}\label{fig1}
\end{center}
\end{figure}

For each sector~$Q_{j}$ with $Q_{j}\cap \{z_{n}\}\ne \varnothing$, pick a point in $Q_{j}\cap \{z_{n}\}$ and name it $z_{j}\in Q_{j}$. Fix $\ell=1,\dotsc,10$ and consider the function
$$
f_{\ell} (z)=f_{\ell,m} (z)=\sum_{j: Q_{j}\in G_{\ell}}
\frac{(1-|z_{j}|^{2})^{1/2}}{1-\overline{z}_{j}z}.
$$
Fix $z_{k} \in Q_{k}\in G_{\ell}$. If $z= e^{i\theta}$,
$|\theta-2\pi k 2^{-m}| < \pi 2^{-m}$, we have $|e^{i\theta} - z_k|
\leq 5 2^{-m}$ and $|e^{i\theta} - z_j| > 8 \pi 2^{-m}$ if $j\ne k$.
Thus
 \begin{equation}\label{eq4.4}
 \begin{split}
 |f'_{\ell}(z)|&\ge \frac{(1-|z_{k}|)^{1/2}|\overline{z}_{k}|} {|1-\overline{z}_{k}z|^{2}}-\sum_{j\ne k} \frac{2(1-|z_{j}|)^{1/2}} {|1-\overline{z}_{j}z|^{2}}\\*[7pt]
 & \ge \frac{2^{-m/2}1/2}{25\; 2^{-2m}} -\sum_{i=1} \frac{2\;2^{-m/2}}{( 8 \pi 2^{-m} i)^{2}} \ge \frac{2^{3m/2}}{100}.
 \end{split}
 \end{equation}

\noindent
Let $H_{\ell}$ be the subfamily of $G_{\ell}$ consisting
of these sectors~$Q_{j}\in G_{\ell}$ with $\{z_{n}\}\cap
Q_{j}\ne\varnothing$. Estimate~\eqref{eq4.4} gives that
\begin{equation}\label{eq4.5}
 \left \{ e^{i\theta} : |f'_{\ell} (e^{i\theta})| >\frac{2^{3m/2}}{10}\right\} \supseteq \bigcup \{ e^{i\theta}: |\theta-2\pi j 2^{-m}|< \pi 2^{-m}\}
 \end{equation}
 where the union is taken over all $j=0,\dotsc,2^{m}-1$ such that $Q_{j}\in H_{\ell}$.
Since \eqref{eq1.1} holds, there exists a constant~$C>0$ such that
\begin{equation}\label{eq4.6}
|\{ e^{i\theta} : |f'_{\ell} (e^{i\theta})| >\lambda\}| \le C\lambda^{-2/3} \|f_{\ell}\|_{2}^{2/3}
\end{equation}
for any $\lambda>0$. Since the points~$\{z_{j}\}$ which appear in
the definition of $f_{\ell}$ form an interpolating sequence with
fixed constants (independent of~$\ell$ and~$m$), we have that
$\|f_{\ell}\|^{2}_{2}$ is comparable to~$\# H_{\ell}$ (see Theorem~B
in \cite{Co}). Taking $\lambda=2^{3m/2} /100$ in~\eqref{eq4.6} and
applying \eqref{eq4.5} we get
$$
2^{-m}\# H_{\ell} \le C_{1} 2^{-m} (\# H_{\ell})^{1/3}.
$$
Hence $\# H_{\ell}\le C_{1}^{3/2}$. Adding over $\ell=1,\dotsc,10$
we deduce $\# \{Q_{j}: Q_{j}\cap \{z_{n}\}\ne\varnothing\} \le 10
C_{1}^{3/2}$.

Let $N_{j}$ be the number of points of $\{z_{n}\}$ contained in
$Q_{j}$. Next we will show that there exists a constant~$K>0$,
independent of~$j$ and~$m$, such that $N_{j}\le K$. Fix $p<1$ and
observe that \eqref{eq1.1} gives that $h' \in H^{2/3}_w \subset
H^{2p/(p+2)}$ for any $h \in (BH^{2})^{\perp}$. So the result of
Cohn (\cite{Co})gives that $B' \in H^p$. In particular $B'$ has
non-tangential limits at almost every point of the unit circle. We
will use an idea of W.~Cohn (see the proof of Theorem~2
in~\cite{Co}) which we collect in the following statement.

\begin{claim}\label{claim}
Assume \eqref{eq1.1} holds. Then for any $h\in (BH^{2})^{\perp}$ and
any $\lambda>0$ one has
$$
|\{ e^{i\theta}: |B' (e^{i\theta}) h(e^{i\theta})|>\lambda\}| \le 3C\left(\frac{\|h\|_{2}}{\lambda}\right)^{2/3}.
$$
\end{claim}

\begin{proof}[Proof of the Claim]
It is well known that any $f\in (BH^{2})^{\perp}$ can be written as
\begin{equation}\label{eq4.7}
f(z)= B(z)\frac{1}{z}\,
\overline{h\left(\frac{1}{\overline{z}}\right)},\quad z\in
\mathbb{C} \setminus \overline{\{z_{k}\}\cup
\left\{\frac{1}{\overline{z}_{k}}\right\}}
\end{equation}
where $h\in (BH^{2})^{\perp}$. Hence
$$
f'(z)= B'(z) \frac{1}{z} \overline{ h\left(
\frac{1}{\overline{z}}\right) } -B(z) \frac{1}{z^{2}}
\,\overline{h\left(\frac{1}{\overline{z}}\right)} -B(z)
\frac{1}{z^{3}} \, \overline{h'\left(\frac{1}{\overline{z}}\right)}.
$$
Since $h\in (BH^{2})^{\perp}\subset H_{w}^{2/3}$, \eqref{eq1.1}
gives that
$$
|\{ e^{i\theta}: |B' (e^{i\theta})h (e^{i\theta})|>\lambda\}| \le 3C \left( \frac{\|h\|_{2}}{\lambda}\right)^{2/3}.
$$
Since the functions~$h$ which arise from functions~$f$ in~\eqref{eq4.7} are dense in $(BH^{2})^{\perp}$, the Claim is proved.
\end{proof}

Fix a sector~$Q_{j}$ of the form $ Q_{j}=\{ re^{i\theta} :
1-2^{-m}\le r<1-2^{-m-1},\, |\theta-2\pi j2^{-m}|<\pi 2^{-m}\} $ and
recall that $N_{j}=\# \{n: z_{n}\in Q_{j}\}$. Pick a point $z_{j}\in
Q_{j}\cap \{z_{n}\}$ and consider the function
$$
h(z) =\frac{(1-|z_{j}|^{2})^{1/2}}{1-\overline{z}_{j}z}.
$$
Note that $h\in (BH^{2})^{\perp}$, $\|h\|_{2}=1$ and $|h
(e^{i\theta})| >2^{m/2} / 10$ if $|\theta-2\pi j 2^{-m}| < \pi
2^{-m}$. Also, if $|\theta-2\pi j 2^{-m}|< \pi 2^{-m}$, we have
$$
e^{i\theta} \frac{B' (e^{i\theta})}{B(e^{i\theta})} =\sum_{n}
\frac{1-|z_{n}|^{2}}{|e^{i\theta}-z_{n}|^{2}}\ge \sum_{n: z_{n}\in
Q_{j}} \frac{1-|z_{n}|^{2}}{|e^{i\theta}-z_{n}|^{2}}\ge
\frac{2^{m}N_{j}}{50}.
$$
Choose $\lambda=2^{\frac{3m}{2}} N_{j} /500 $ to obtain
$$
\{ e^{i\theta}: |\theta-2\pi j 2^{-m}| <\pi 2^{-m} \} \subseteq
\{e^{i\theta}:|B' (e^{i\theta})h(e^{i\theta})|>\lambda\}.
$$
Applying the Claim we deduce
$$
\pi 2^{-m} \le
C\left(\frac{500}{N_{j}2^{\frac{3m}{2}}}\right)^{2/3}.
$$
Hence
$$
N_{j}\le 500 C^{3/2}.
$$
This finishes the proof.
\end{proof}

\begin{obs}
The only Blaschke products~$B$ for which $f'\in H^{2/3}$ for any $f\in (BH^{2})^{\perp}$ are the finite ones.
\end{obs}
\begin{proof} We argue by contradiction. Let $B$ be an infinite Blaschke product
such that $f'\in H^{2/3}$ for any $f\in (BH^{2})^{\perp}$. Let
$\{z_{n}\}$ be the zeros of~$B$. Taking a subsequence if necessary,
we can assume that $\{z_{n}\}$ is an interpolating sequence. Arguing
as in the previous Claim, one gets that
\begin{equation}\label{eq4.8}
B'h\in H^{2/3}\text{ for any }h\in (BH^{2})^{\perp}.
\end{equation}
Pick a sequence~$\{w_{n}\}$ of complex values such that
\begin{align*}
&\sum |w_{n}|^{2} (1-|z_{n}|)<\infty,\\
&\sum |w_{n}|^{2/3} (1-|z_{n}|)^{1/3}=\infty.
\end{align*}
Since $\{z_{n}\}$ is an interpolating sequence, one can choose $h\in (BH^{2})^{\perp}$ such that $h(z_{n})=w_{n}$, $n=1,2,\dotsc$ . Since $\sum(1-|z_{n}|)\delta_{z_{n}}$ is a Carleson measure, from \eqref{eq4.8} one deduces
$$
\sum |B'(z_{n}) h(z_{n})|^{2/3} (1-|z_{n}|)<\infty.
$$
Since $\{z_{n}\}$ is an interpolating sequence, one has $\inf\limits_{n} |B'(z_{n})| (1-|z_{n}|)>0$. Hence
$$
\sum |w_{n}|^{2/3} (1-|z_{n}|)^{1/3}<\infty
$$
which gives the contradiction.
\end{proof}

\end{document}